\pdfoutput=1
\documentclass[11pt,letterpaper]{amsart} 

\usepackage{comment}
\usepackage{ifluatex}
\usepackage[english]{babel} 

\usepackage{amscd,amsmath,amssymb,amsthm} 
\usepackage{array}                        
\usepackage{multirow}                     

\ifluatex
  \usepackage{fontspec}
\else
  \usepackage[utf8]{inputenc} 
  \usepackage{lmodern}
  \usepackage[T1]{fontenc}
\fi

\usepackage[margin=0.9in]{geometry}    

\usepackage{mathtools}
\usepackage{bm} 
\usepackage{dutchcal} 

\usepackage{slashed} 
\usepackage[babel=true]{microtype}
\usepackage[autostyle=true]{csquotes} 
\usepackage[dvipsnames]{xcolor}
\usepackage[biblatex=true]{embrac}

\usepackage[%
bookmarks=true,
colorlinks,
linkcolor=blue,
urlcolor=blue,
citecolor=blue,
plainpages=false,
pdfpagelabels,
final,
breaklinks=true,
pdfusetitle,
]{hyperref}
\usepackage[citestyle=numeric-comp,
            bibstyle=numeric-comp,
            backend=biber,
            url=false,
            doi=true,
            isbn=false,
            giveninits=true,
            maxbibnames=100,
            sorting=nyt]{biblatex}

\AtEveryBibitem{\clearfield{month}}
\AtEveryBibitem{\clearfield{day}}
\AtEveryBibitem{%
  \ifentrytype{thesis}
    {}
    {
      \ifentrytype{online}{}
      {
      \clearfield{url}%
      \clearfield{urldate}%
      }
    }%
}
            
\usepackage{tikz}
\usetikzlibrary{cd} 

\usepackage[all]{xy}

\usepackage[shortlabels]{enumitem}
\newlist{myenumi}{enumerate}{1}
\setlist[myenumi,1]{label=\upshape(\roman*)}
\newlist{myenuma}{enumerate}{1}
\setlist[myenuma,1]{label=\upshape(\alph*)}
\usepackage{color}
\usepackage[textwidth=0.8in]{todonotes}
\usepackage{thmtools}
\declaretheorem[name=Theorem, numberwithin=section]{theorem}
\declaretheorem[name=Theorem, numbered=no]{theorem*}
\declaretheorem[name=Lemma,numberlike=theorem]{lemma}
\declaretheorem[name=Lemma,numbered=no]{lemma*}
\declaretheorem[name=Corollary,numberlike=theorem]{cor}
\declaretheorem[name=Proposition,numberlike=theorem]{prop}

\declaretheorem[name=Example, numberlike=theorem, style=remark]{example}
\declaretheorem[name=Remark, numberlike=theorem, style=remark]{rem}

\declaretheorem[name=Theorem]{thmx}
\declaretheorem[name=Corollary, numberlike=thmx]{corx}

\numberwithin{equation}{section}
\allowdisplaybreaks[1]
\usepackage[noabbrev]{cleveref} 
\crefname{figure}{Figure}{Figures}
\crefname{table}{Table}{Tables}
\crefname{theorem}{Theorem}{Theorems}
\crefname{thmx}{Theorem}{Theorems}
\crefname{lemma}{Lemma}{Lemmas}
\crefname{definition}{Definition}{Definitions}
\crefname{setup}{Setup}{Setups}
\crefname{conjecture}{Conjecture}{Conjectures}
\crefname{question}{Question}{Questions}
\crefname{cor}{Corollary}{Corollaries}
\crefname{corx}{Corollary}{Corollaries}
\crefname{prop}{Proposition}{Propositions}
\crefname{example}{Example}{Examples}
\crefname{rem}{Remark}{Remarks}
\crefname{section}{Section}{Sections}
\crefname{subsection}{Subsection}{Subsections}
\crefname{chapter}{Chapter}{Chapters}
\crefname{appendix}{Appendix}{Appendices}
\crefdefaultlabelformat{#2\textup{#1}#3}
\creflabelformat{enumi}{(#2#1#3)}

\usepackage{crossreftools}
\usepackage{xparse}

\usepackage{ourmacros}

\newcommand{\diag}{\operatorname{diag}}

\newcommand{\incl}{\mathrm{incl}}

\title{Homotopical Robustness of Isometries on the Cayley Plane}

\hypersetup{
  pdfauthor={Thorsten Hertl}
  pdfauthor={Isla Lim}
}

\author{Thorsten Hertl}
\address[T.~Hertl]{School of Mathematics and Statistics, The University of Melbourne, Australia }
\email{\href{mailto:thorsten.hertl@unimelb.edu.au}{thorsten.hertl@unimelb.edu.au}}
\urladdr{\href{https://thorsten-hertl.github.io/}{https://thorsten-hertl.github.io/}}

\author{Isla Lim}
\address[I.~Lim]{School of Mathematics and Statistics, The University of Melbourne, Australia }
\email{\href{mailto:yellowfin1728@gmail.com}{yellowfin1728@gmail.com}}

\date{\today}

\begin{document}

\begin{abstract}
  We compute the effect of the inclusion $\Ffour \rightarrow \hAut(\OP^2)$ on rational homotopy groups by determining the rational homotopy class of the isometry action $\Ffour \times \OP^2 \rightarrow \OP^2$ in terms of a homomorphism between the minimal models of source and target. 
  Although special emphasis is put on the Cayley plane, our results generalise to all simply connected rank $1$-symmetric spaces.
\end{abstract}

\maketitle



\section{Introduction}\label{Section - Introduction}

Let $\OP^2$ be the Cayley plane, which is defined to be the subspace of $\mathrm{Herm}(3,\Oct)$ consisting of all hermitian $(3\times 3)$-matrices with entries in the octonions $\Oct$ that are rank one projections.
Embed $\mathrm{Herm}(1,\Oct) = \Oct$ and $\mathrm{Herm}(2,\Oct)$ into the upper left corner $\mathrm{Herm}(3,\Oct)$, that is via the map $A \mapsto \diag(A,0)$, and set $\OP^m = \OP^2 \cap \mathrm{Herm}(m+1,\Oct)$ for $m=0,1$.

The compact Lie group $\Ffour$ acts faithfully and transitively on $\OP^2$ with stabiliser $\mathrm{Stab}(\OP^0) = \Spin(9)$ and $\mathrm{Stab}(\OP^1) = \{g \in \Ffour \, : \, g\cdot x=x\, \text{ for all } x\in \OP^1\}= \Spin(7)$, see \cite{Harvey1990SpinorsAndCalibrations}, especially Theorem 14.79 paired with the definition of the $\Spin(9)$ action in the proof of Theorem 14.99. 
In fact, $\OP^2$ can be equipped with a Riemannian metric (unique up to scaling) that turns it into a symmetric space and $\Ffour$ will then be the path component $\mathrm{Iso}(\OP^2)$ of the identity inside the group of Riemannian isometries, see \cite[p.287]{Wolf2011SpacesConstCurv}.
The left action therefore gives rise to an inclusion
\begin{equation*}
    s \colon \Ffour = \mathrm{Iso}(\OP^2)_0 \rightarrow \hAut(\OP^2)_0 \quad \text{ and } \quad s\colon \Ffour/\Spin(7) \rightarrow \mathrm{C}(\OP^1,\OP^2),
\end{equation*}
where $\mathrm{C}(X,Y)$ is the topological space of all continuous maps from $X$ to $Y$ equipped with the compact open topology, and $\hAut(X) \subseteq \mathrm{C}(X,X)$ is the subspace consisting of all homotopy equivalences and $\hAut(\OP^2)_0$ its path component of the identity.

Our first result describes the effect of $s$ on rational homotopy groups.
\begin{thmx}\label{Main Theorem - Rational Homotopy Groups Cayley}
  The homomorphisms
  \begin{align*}
      \pi_r(s) \colon \pi_r(\Ffour)_\Q &\rightarrow \pi_r(\hAut(\OP^2),\id)_\Q, \\
      \pi_r(s) \colon \pi_r(\Ffour/\Spin(7))_\Q &\rightarrow \pi_r(\mathrm{C}(\OP^1,\OP^2),\incl)_\Q,
  \end{align*}
  are isomorphisms if $r>11$. 
  The target groups are zero if $r \leq 11$.
\end{thmx}
From the knowledge of the rational homotopy groups of $\Ffour$ and $\Ffour/\Spin(7)$ we deduce:
\begin{corx}\label{Main Cor: Rational Homotopy Groups Mapping Spaces}
    The rational homotopy groups of $\hAut(\OP^2)$ and $\mathrm{C}(\OP^1,\OP^2)$ are given by
    \begin{equation*}
        \pi_r\bigl(\hAut(\OP)^2,\id\bigr)_\Q = \pi_r\bigl(\mathrm{C}(\OP^1,\OP^2),\incl\bigr)_\Q = \begin{cases}
            \Q, & \text{if } r=15,23, \\
            0, & \text{otherwise.}
        \end{cases}
    \end{equation*}
\end{corx}
Although Corollary \ref{Main Cor: Rational Homotopy Groups Mapping Spaces} can be easily computed from Sullivan's technique \cite{Sullivan1977Infinitesimal} and its generalisation to arbitrary mapping space in \cite{Buijs2008RationalLieAlgebraFunction}, these computations do not provide the information that the non-trivial elements arise from $\pi_{15}(\Ffour)_\Q$ and $\pi_{23}(\Ffour)_\Q$. 

Similar results have been by established Meier and Strebel for $\RP^n$ \cite{Meier1981RP}, by Sasao for $\CP^n$ \cite{Sasao1974CP}, and by Yamaguchi for $\HP^n$ \cite{Yamaguchi1983HP}.
However, their proof strategies are not applicable to the Cayley plane.
Indeed, Meier and Strebel rely on the classification of $\Q$-acyclic spaces with $\Q$-acyclic fundamental groups, while Sasao and Yamaguchi use the fibrations
\begin{equation*}
    S^1 \rightarrow S^{2n+1} \rightarrow \CP^n \quad \text{ and } \quad S^3 \rightarrow S^{4n+3} \rightarrow \HP^n
\end{equation*}
in an essential manner.
The Cayley plane, however, does not exhibit an analogous fibration because the octonions fail to be associative, which forces us to use a proof strategy that is entirely different from theirs.

In contrast to the rather topological approach of Sasao and Yamaguchi, we make use of the full-fledged machinery of rational homotopy theory by transferring the problem into the set-up of Sullivan algebras, where it becomes essentially an (linear) algebraic problem that can be solved in a fairly hands-on manner.
This approach has the additional advantage that all simply connected rank-1 symmetric spaces can be treated in a unifying theme, which enables us to provide new proofs for the results in \cite{Sasao1974CP} and \cite{Yamaguchi1983HP} as well.

Furthermore, our approach allows to derive the stronger result that describes the rational homotopy class of the symmetry group action on projective spaces.
To formulate this result in the special case of the action $\ell \colon \Ffour \times \OP^2 \rightarrow \OP^2$ (for the more general case we refer to Proposition \ref{prop: left action model projectives spaces} and Proposition \ref{prop: left action model spheres}), let $\mathsf{M}_{\OP^2} = \Lambda[a_{8},b_{23} \, | \, db_{23} = a_{8}^{3}]$ be the minimal model of $\OP^2$ in the sense of rational homotopy theory, and recall that the rational cohomology ring $H(\Ffour) = \Lambda[\xi_3,\xi_{11},\xi_{15},\xi_{23}]$ is an exterior algebra.
\begin{thmx}\label{Main Theorem - Homotopy Class Left Action}
    The rational model $\mathsf{M}(\ell) \colon \mathsf{M}_{\OP^2} \rightarrow H(\Ffour) \otimes \mathsf{M}_{\OP^2}$ of the action $\ell \colon \Ffour \curvearrowright \OP^2$ is the (homotopy class of the) homomorphism of commutative, differential, graded algebras determined by
    \begin{equation*}
        \mathsf{M}(\ell)(a_8) = 1 \otimes a_8 \quad \text{ and } \quad \mathsf{M}(\ell)(b_{23}) = 1 \otimes b_{23}  +  \xi_{15} \otimes a_8 + \xi_{23} \otimes 1. 
    \end{equation*}
    Moreover, the coefficients $\xi_{15}$ and $\xi_{23}$ are homotopy invariant in the sense that they remain unchanged if one changes $\mathsf{M}(\ell)$ within its homotopy class.
\end{thmx}
As a consequence, the rational homotopy class of $s \colon \Ffour \rightarrow \hAut(\OP^2)$ is completely determined by its induced homomorphisms on rational homotopy groups (which is not clear at first).

\vspace{10pt}
\noindent
\textbf{Outline of the article:}
Section \ref{Section: Preliminaries} provides a concise recollection on rational homotopy theory needed in this article. 
Subsection \ref{subsection: RHT} presents the abstract theory while Subsection \ref{subsection: examples of models} provides the rational models for the examples we study in Section \ref{Section: Robustness}. 
It serves further as an opportunity to introduce our notational conventions.
Section \ref{Section: Robustness} is reserved for the proofs of Theorem \ref{thm: general main result rhg projective spaces} and \ref{thm: general main result rhg spheres} that imply Theorem \ref{Main Theorem - Rational Homotopy Groups Cayley} and Proposition \ref{prop: left action model spheres} and \ref{prop: left action model projectives spaces} that imply Theorem \ref{Main Theorem - Homotopy Class Left Action}.

\vspace{10pt}
\noindent
\textbf{Background Information:}
This paper is based on the MSc thesis of the second author, entitled `Non-triviality of $B\mathrm{Iso}(\OP^2) \rightarrow B\hAut(\OP^2)$', which was submitted to the University of Melbourne in May 2026 under the supervision of the first author and Diarmuid Crowley.

\vspace{10pt}
\noindent
\textbf{Acknowledgments:} 
T.~Hertl's work was funded by the Australian Research Council Discovery Project DP220102163. 
The two authors would like to thank Diarmuid Crowley for his helpful remarks and his interest in this work.

\section{Preliminaries}\label{Section: Preliminaries}

\subsection{Rational Homotopy Theory}\label{subsection: RHT}

We start with a concise recollection of the required basics on rational homotopy theory followed by some sample calculations in the next subsection.
More details can be found in \cite{Bousfield1976PLdeRham} and \cite{Felix2001RHT}.
We first discuss rational spaces and then look at their algebraic counterparts.
\subsubsection{Rational Spaces}
Recall that a path-connected topological space $X$ is called \emph{nilpotent} if its fundamental group is nilpotent and if the action of $\pi_1(X)$ on $\pi_n(X)$ is nilpotent, see Chapter II of \cite{Hilton1975Localisation}.
A nilpotent space is called \emph{rational} if, for all $n$ and $r \in \N$, the $n$-th power map $(\placeholder)^n \colon \pi_r(X) \rightarrow \pi_r(X)$ is a bijection (an isomorphism if $r > 1)$.
The \emph{rationalisation} of a topological space $X$ is a pair $(X_\Q,\mathsf{loc}_\Q)$ consisting of a rational space $X_\Q$ and a continuous map $\mathsf{loc}_\Q \colon X \rightarrow X_\Q$ such that precomposition with $\mathsf{loc}_\Q$ induces a bijection
\begin{equation*}
    (\placeholder) \circ \mathsf{loc}_\Q \colon [X_\Q,Y] \xrightarrow{\cong} [X,Y]
\end{equation*}
for all nilpotent, rational spaces $Y$.
Rationalisations exists for every nilpotent CW-complex, see Theorem 3A in Chapter II of \cite{Hilton1975Localisation}, and their universal property implies that they are unique up to homotopy equivalence.
Moreover, $\pi_r(\mathsf{loc}_\Q) \colon \pi_r(X)_\Q \rightarrow \pi_r(X_\Q)$ induces an isomorphism\footnote{For $\pi_1(X)$, we use the Mal'cev-completion, which agrees with the tensor product if $\pi_1(X)$ is abelian, see Chapter I of \cite{Hilton1975Localisation} for details.}, see Theorem 3B in Chapter II of \cite{Hilton1975Localisation}.

A continuous map $f \colon X \rightarrow Y$ between two path-connected nilpotent topological spaces is called a \emph{rational homotopy equivalence} if it induces an isomorphism between their rationalised homotopy groups.
Two spaces are called \emph{rational homotopy equivalent} if there exists a zig-zag $X \rightarrow Z_1 \leftarrow \dots \leftarrow Z_n \rightarrow Y$ of rational homotopy equivalences between them. Equivalently, $X$ and $Y$ are rational homotopy equivalent if their rationalisations are weak homotopy equivalent.

The reason why we consider the larger class of nilpotent spaces (instead of simply connected spaces) is that this subclass of topological spaces is closed under taking mapping spaces:
If $X$ is a finite CW complex and $Y$ a nilpotent CW complex, then each path-component of the space of continuous maps $\mathrm{C}(X,Y)$ is nilpotent, see Theorem 2.5 in Chapter II of \cite{Hilton1975Localisation}.
For our purpose, the following result is essential.
\begin{theorem}{\cite[Theorem 3.11]{Hilton1975Localisation}}\label{thm: localisation of mapping spaces}
    Let $X$ be a finite, connected CW-complex, $Y$ be a nilpotent CW-complex, and $\mathsf{loc}_\Q \colon Y \rightarrow Y_\Q$ its rationalisation.
    Then postcomposition with $\mathsf{loc}_\Q$ induces a rationalisation  $\mathsf{loc}_\Q \circ (\placeholder) \colon \mathrm{C}(X,Y;f) \rightarrow \mathrm{C}(X,Y_\Q,\mathsf{loc}_\Q \circ f)$, where $\mathrm{C}(X,Y;f)$ is the path component of $\mathrm{C}(X,Y)$ that contains $f \colon X \rightarrow Y$.
\end{theorem}
The advantage of rational spaces is that their homotopy types and the set of their homotopy classes can be completely described in terms of algebraic data.
\subsubsection{Sullivan Algebras}
The algebraic counterpart for rational topological spaces will be the category of commutative, differential, (cohomologically) graded, (unital) algebras (cdgas in short) over the field $\Q$ that are concentrated in non-negative degrees.
The assignment sending a topological space $X$ to $\Omega_{PL}(S(X))$, the cdga of polynomial differential forms on the singular set $S(X)$, yields a contravariant functor $\mathsf{Top} \rightarrow \mathsf{CDGA}$.
The cohomology $H(\Omega_{PL}(X))$ is isomorphic to the singular cohomology $H_{sing}(X;\Q)$ of the underlying topological space $X$, see \cite{Sullivan1977Infinitesimal} or \cite[Theorem 10.9]{Felix2001RHT}.

The role of CW-complexes inside the category $\mathsf{CDGA}$ are taken by \emph{Sullivan algebras}.
Recall that a Sullivan algebra $(\Lambda V,d)$ is the free algebra of a non-negatively graded, rational vector space $V$ such that $d$ satisfies the \emph{nilpotency} condition, which means that there is a filtration $V(0) \subseteq V(1) \subseteq \dots$ such that $\bigcup V(k) = V$, such that $d=0$ on $V(0)$, and $d \colon V(n) \rightarrow \Lambda V(n-1)$ for all $n\geq 1$.
A Sullivan algebra is called \emph{minimal} if $d(v)$ is quadratic, that is $d(V) \subseteq \Lambda^{\geq2} V$.
A homomorphism of commutative, differential, graded algebras (abbreviated to dga-homomorphism in the future) $\varphi \colon (A,d) \rightarrow (B,d)$ is a \emph{quasi-isomorphism} if its induced homomorphism on cohomology $H(\varphi) \colon H(A,d) \rightarrow H(B,d)$ is an isomorphism.

An important example of a (non-minimal) Sullivan algebra is the \emph{algebraic interval} $I = \Lambda[t,dt]$, with $t$ a variable in degree zero and the differential satisfying the tautological relation $d(t) = dt$ and $d(dt)=0$. 
It comes with two dga-homomorphisms $\ev_0,\ev_1 \colon \Lambda[t,dt] \rightarrow \Q$ that are completely described by $\mathrm{ev}_j(t) = j$ for $j=0,1$. 
A \emph{homotopy} between two dga-homomorphisms $\varphi_0,\varphi_1 \colon A \rightarrow B$ is a dga-homomorphism $H \colon A \rightarrow I \otimes B$ such that $\ev_j \circ H = \varphi_j$.
Being homotopic is an equivalence relation on the set $\mathrm{Hom}_{dga}(\Lambda V,A)$ if the domain is a Sullivan algebra, see \cite[Proposition 6.3]{Bousfield1976PLdeRham}.
Moreover, there is a form of Whitehead's theorem in the sense that postcomposition with quasi-isomorphisms $\varphi \colon (A_1,d) \rightarrow (A_2,d)$ induces a bijection between homomotopy classes $[\Lambda V,A_1] \rightarrow [\Lambda V,A_2]$, see \cite[Proposition 5.7]{Bousfield1976PLdeRham}.

A \emph{Sullivan model} of a topological space $X$ is a pair of a Sullivan algebra $(\Lambda V,d)$ and a quasi-isomorphism $m \colon \Lambda V \rightarrow \Omega_{PL}(X)$. 
If $(\Lambda V,d)$ is minimal, then we refer to it as a \emph{minimal model} of $X$ and denote it with $\mathsf{M}_X$.
The notation is justified, because minimal models exist for path-connected topological spaces \cite[Proposition 14.3]{Felix2001RHT} and are unique up to isomorphism \cite[Theorem 14.12]{Felix2001RHT}.
A \emph{rational model} for a continuous map $f \colon X \rightarrow Y$ is a dga-homomorphism between Sullivan models of domain and target $\varphi \colon \Lambda V_Y \rightarrow \Lambda V_X$ such that $m_X \circ \varphi$ is homotopic to $\Omega_{PL}(f) \circ m_Y$.
By Whiteheads theorem, the homotopy class of this lift is unique. 
We will denote the lift between two minimal models by $\mathsf{M}(f)$ (and do not notationally distinguish the homotopy class from any of its representatives).

It was proved in \cite[Theorem 4.3]{Bousfield1976PLdeRham} that $\mathsf{CDGA}$ has model category structure with weak-equivalences being quasi-isomorphisms and fibrations being surjective dga-homomorphism and that Sullivan algebras are cofibrant objects.
The authors of \cite{Bousfield1976PLdeRham} further show, see Chapter 9 in loc. cit. that the assignment $X \mapsto \mathsf{M}_X$ gives rise to a Quillen adjunction
\begin{equation*}
    \langle \, \cdot \, \rangle \colon \mathsf{CDGA} \rightleftarrows \mathsf{Top} : \mathsf{M}_{(\placeholder)}
\end{equation*}
that induces an equivalence of categories
\begin{equation*}
    \mathrm{fin}_\Q\text{-}\mathrm{Ho}(\mathsf{CDGA}) \overset{\cong}{\rightleftarrows} \mathrm{fin}_{\Q,\mathrm{nil}}\text{-}\mathrm{Ho}(\mathsf{Top}) 
\end{equation*}
where $\mathrm{fin}_\Q\text{-}\mathrm{Ho}(\mathsf{CDGA})$ is the full subcategory of $\mathrm{Ho}(\mathsf{CDGA})$ whose objects are quasi-isomorphic to a minimal algebra $\Lambda V$ with a generating graded vector space that is finite dimensional in each degree.
On the other hand, $\mathrm{fin}_{\Q,\mathrm{nil}}\text{-}\mathrm{Ho}(\mathsf{Top})$ is the full subcategory of $\mathrm{Ho}(\mathsf{Top})$ whose objects are weakly equivalent to a rational nilpotent space of finite $\Q$-type, that is, its rational homology is finite dimensional in each degree. 
In particular, we have the following consequence, see \cite[Proof 11.9]{Bousfield1976PLdeRham} for second part of the following statement.
\begin{theorem}\label{thm: homotopy classification}
    For every two rational, nilpotent spaces of finite $\Q$-type $X,Y$, we have a bijection of homotopy classes
\begin{equation}
    [X,Y] \cong [\mathsf{M}_Y,\mathsf{M}_X] \quad \text{ given by } \quad [f] \mapsto [\mathsf{M}_f].
\end{equation}
   Furthermore, if the minimal model $\mathsf{M}_Y = \Lambda V_Y$ is generated by $V_Y$, then there is a natural bijection $\pi_k(Y)_\Q \cong \mathrm{Hom}(V^k_Y,\Q)$, which is an isomorphism whenever the domain is abelian.
\end{theorem}
For a dga-homomorphism $\varphi\colon A \rightarrow B$ between two cdga and $r \in \Z$, denote by $\Der_r^\varphi(A,B)$ the vector space of all $\varphi$-derivations that lower the degree by $r$, that is, the vector space of all linear maps $\theta \colon A \rightarrow B$ that lower the degree by $r$ and satisfy $\theta(a_1 \cdot a_2) = \theta(a_1)\varphi(a_2) + (-1)^{r\cdot\mathrm{deg}(a_2)}\varphi(a_1)\cdot \theta(a_2)$.
It is straightforward to verify that, for $\theta \in \Der_r^\varphi(A,B)$, the linear map $\delta_r(\theta) := d\theta - (-1)^r \theta d$ is a $\varphi$-derivation lowering the degree by $1$ and that $\delta_{r-1} \circ \delta_r =0$.
In particular, $\delta$ turns $\Der^\varphi(A,B) = \bigoplus_{r\in \Z} \Der_r^\varphi(A,B)$ into a homologically graded chain complex.
In \cite{Sullivan1977Infinitesimal}, Sullivan provided an algebraic model for the rational homotopy groups of the homotopy automorphisms $\mathrm{hAut}(X)$ in terms of derivations on the minimal model of $X$.
The next result, proved in \cite{Buijs2008RationalLieAlgebraFunction}, generalises the one of Sullivan \cite{Sullivan1977Infinitesimal} to arbitrary mapping spaces.
\begin{theorem}[Theorem 1 in \cite{Buijs2008RationalLieAlgebraFunction}]\label{Thm: Mapping Spaces via derivations}
    Let $X,Y$ be nilpotent CW complexes with $X$ finite and $Y$ of finite $\Q$-type.
    Then there are isomorphisms of $\Q$-vector spaces
    \begin{align*}
        \pi_r(\mathrm{C}(X,Y),f)_\Q &\cong \Der_r^{\mathsf{M}(f)}( \mathsf{M}_Y, \mathsf{M}_X , \delta), \\
        \Gamma\pi_1(\mathrm{C}(X,Y),f)_\Q &\cong \Der_1^{\mathsf{M}(f)}(\mathsf{M}_Y,\mathsf{M}_X,\delta),
    \end{align*}
    where $\Gamma\pi_1(\mathrm{C}(X,Y),f)_\Q$ denotes the rational vector space $\bigoplus_j \Gamma_j/\Gamma_{j+1} \otimes \Q$ associated to the lower central series $\Gamma_0 = \pi_1(\mathrm{C}(X,Y),f) \supseteq \Gamma_1 \supseteq \dots \supseteq \Gamma_n = \{1\}$ of the nilpotent group $\pi_1(\mathrm{C}(X,Y),f)$.
\end{theorem}

\subsection{Examples of Rational Models}\label{subsection: examples of models}

We work out the rational models for the spaces we consider Section \ref{Section: Robustness}.
Unless stated otherwise, rational coefficients are understood.
\begin{example}\label{Exmpl: Odd Spheres}
    The minimal model of $S^n$ with $n = 2k+1$ is given by $\mathsf{M}_{S^n} = \Lambda[a_n \, | da_n = 0]$.
    It is easy to see that a map $m \colon \mathsf{M}_{S^n} \rightarrow \Omega_{PL}(S^n)$ that sends $a_n$ to a generator of the volume form $\vol_{S^n} = H^n(S^n) = H^n(\Omega_{PL}(S^n),d)$ is a quasi-isomorphism.
\end{example}
\begin{example}\label{Exmpl: Even Dimensional Spheres}
    The minimal model of $S^n$ with $n = 2k$ is given by $\mathsf{M}_{S^n} = \Lambda[a_n, b_{2n-1} \, | db_{2n-1} = a_n^2]$.
    A model map is inductively constructed as follows: Send $a_n$ to a generator $\omega_n$ of the volume form $\vol_{S^n} \in H^n(S^n)$.
    We know that $\omega^2_n$ must be exact, so we send $b_{2n-1}$ to a form $\eta$ with $d\eta = \omega_n^2$.
    By construction, the map $m$ extends to a quasi-isomorphism $m\colon \mathsf{M}_{S^n} \rightarrow \Omega_{PL}(S^n)$.
\end{example}
\begin{example}\label{Exmpl: Complex Projective Spaces}
    The cohomology ring of $\CP^n$ is well known: $H(\CP^n) = \Q[a_2]/\langle a_2^{n+1} \rangle$.
    Therefore, the minimal model of $\CP^n$ is given by $\mathsf{M}_{\CP^n} = \Lambda[a_2,b_{2n+1} \, | \, db_{2n+1} = a_{2}^{n+1}]$.
    A model map $m \colon \mathsf{M}_{\CP^n} \rightarrow \Omega_{PL}(\CP^n)$ sends $a_2$ to a generator $\alpha_2$ of $a_2 \in H^2(\CP^n)$ and $b_{2n+1}$ to a primitive of $\alpha_{2}^{n+1}$.

    Since the $\CP^n$ can be given a CW-structure such that the $2m$-skeleton $(\CP^n)^{(2m)}$ can be identified with $\CP^m$ it follows that the inclusion $\CP^m \hookrightarrow \CP^n$ induces the canonical projection $\Q[a_2]/\langle a_2^{n+1}\rangle \rightarrow \Q[a_2]/\langle a_2^{m+1}\rangle$ on cohomology.
    It is not hard to see that there is one and only one dga-homomorphism $\mathsf{M}_{\CP^n} \rightarrow \mathsf{M}_{\CP^m}$ that induces the aforementioned ring homomorphism on cohomology, namely 
    \begin{align*}
        \mathsf{M}_{\CP^n} = \Lambda[a_2,b_{2n+1} \, | \, db_{2n+1} = a_2^{n+1}] &\rightarrow \Lambda[a_2,b_{2m+1}\, | \, db_{2m+1} = a_2^{m+1}] = \mathsf{M}_{\CP^m}, \\
        a_2 &\mapsto a_2, \quad b_{2n+1} \mapsto a_{2}^{n-m}b_{2m+1}.
    \end{align*}
\end{example}
\begin{example}\label{Exmpl: Quaternionic Porjective Spaces}
    Completely analogously to the previous example, we derive
    \begin{align*}
        \mathsf{M}_{\HP^n} = \Lambda[a_4,b_{4n+3} \, | \, db_{4n+3} = a_{4}^{n+1}] &\rightarrow  \Lambda[a_4,b_{4n+3} \, | \, db_{4m+3} = a_{4}^{m+1}] = \mathsf{M}_{\HP^m},  \\
        a_4 &\mapsto a_4, \quad b_{4n+3} \mapsto a_4^{n-m}b_{4m+3}.
    \end{align*}
\end{example}
\begin{example}\label{Exmpl: Cayley Plane}
    It is well known that $\OP^1 \cong S^8$ and that $\OP^2$ has a CW-structure of the form $\{\ast\} = \OP^0 \subset \OP^1 = (\OP^2)^{(8)} \subseteq \OP^2 = (\OP^2)^{(16)}$. 
    In particular, its cohomology ring is given by $H(\OP^n) = \Q[a_8]/\langle a_8^{n+1}\rangle$ for $n \leq 2$.
    Arguing as in Example \ref{Exmpl: Complex Projective Spaces}, we deduce that the inclusion ${\OP^1} \hookrightarrow \OP^2$ is modelled by
    \begin{align*}
        \mathsf{M}_{\OP^2} = \Lambda[a_8,b_{23} \, | \, db_{23} = a_2^3] &\rightarrow \Lambda[a_8,b_{15} \, | \, db_{15} = a_2^2] = \mathsf{M}_{\OP^1} \\
        a_8 &\mapsto a_8, \quad b_{23} \mapsto a_8b_{15}.
    \end{align*}
\end{example}

By a theorem of Hopf the rational cohomology of a compact Lie group is a free exterior algebra, see \cite[Theorem 1.34]{Felix2008AlgebraicModelsGeometry}, so its cohomology group agrees with its minimal model.
To set up our notational conventions, we recall the cohomology rings of the Lie groups of our interests, see Example 3.37, and Example 3.40-3.42 in  \cite{Felix2008AlgebraicModelsGeometry} for the equivalent statements of their classifying spaces.
\begin{example}\label{Exmpl: Minimal model of Lie groups}
 $ $
    \begin{itemize}
        \item[(i)] If $n = 2k$ is even, then $\mathsf{M}_{\SOrth(n)} = H(\SOrth(n)) = \Lambda[\pi_{3},\dots,\pi_{4k-5},\varepsilon_{2k-1}]$. 
        Note in particular that $\mathsf{M}_{\SOrth(2)} = \Lambda[\varepsilon_1]$.
        \item[(ii)] If $n=2k+1$ is odd, then $\mathsf{M}_{\SOrth(n)} = H(\SOrth(n)) = \Lambda[\pi_{3},\dots,\pi_{4k-1}]$ 
        \item[(iii)] For all $n \geq 1$ we have $\mathsf{M}_{\U(n)} = H(\U(n)) =  \Lambda [\gamma_1,\dots,\gamma_{2n-1}]$.
        \item[(iv)] For all $n \geq 1$ we have $\mathsf{M}_{\mathrm{Sp}(n)} = H(\Sp(n)) = \Lambda[q_3,\dots, q_{4n-3}]$.
        \item[(v)] Since $\Spin(9) \rightarrow \mathrm{SO}(9)$ is a two-sheeted cover, the two groups are rational homotopy equivalent and we have $\mathsf{M}_{\Spin(9)} = \Lambda[\pi_3,\pi_7,\pi_{11},\pi_{15}]$. 
        Similarly we deduce that $\mathsf{M}_{\Spin(7)} = \Lambda[\pi_3,\pi_7,\pi_{11}]$.
        \item[(vi)] Borel \cite{Borel1953CohomologyCompactLie} computed the integral cohomology ring of $\mathrm{F}_4$. 
        By tensoring with the rational numbers, we deduce that $\mathsf{M}_{\Ffour} = H(\Ffour) =  \Lambda[ \xi_3,\xi_{11},\xi_{15},\xi_{23}]$.
    \end{itemize}
    It is well known that the inclusions of these groups into each other induce the obvious homomorphisms between their homology groups, for example the inclusion $\SOrth(2k-1) \rightarrow \SOrth(2k)$ induces the algebra homomorphism $\Lambda[\pi_3,\dots,\pi_{4k-5},\varepsilon_k] \rightarrow \Lambda[\pi_3,\dots,\pi_{4k-5}]$ that sends $\pi_j$ to $\pi_j$ and $\varepsilon_{2k-1}$ to zero.
    Moreover, the homomorphism induced by the inclusion $\Spin(9) \rightarrow \Ffour$ sends $\xi_{23}$ to zero and $\xi_j \mapsto \pi_j$ for $j=3,11,15$.
    
\end{example}
%
Next we would like to understand the rational model of the quotient maps.
We start with a technical lemma that classifies homotopies between certain Sullivan algebras.
\begin{lemma}\label{lem: rigid homotopy classes}
    For $k\geq 1$ and $n\geq 2$, consider a minimal Sullivan algebra of the form $(\Lambda V,d)= \Lambda[a_{2k},b_{2nk-1} \, | \, db_{2kn-1} = a_{2k}^n]$. 
    Let $(A,0)$ be a cdga with trivial differential that are generated by elements of odd degree.
    If $n\geq 3$,  then each homotopy class $(\Lambda V,d) \rightarrow (A,0)$ has a unique representative.
    If $n=2$, then every dga-homomorphism $\varphi \colon (\Lambda V,d) \rightarrow \mathsf{A}$ with $\varphi(a_{2k})=0$ is the unique representative of its homotopy class.
\end{lemma}
\begin{proof}
    Each homotopy $H$ is uniquely determined its image of the generators.
    Degree reasons imply
    \begin{align*}
        H(a_{2k}) ={}& 1 \otimes x_{2k} + t \otimes y_{2k} + dt \otimes z_{2k}, \qquad \text{ and} \\
        H(b_{2kn-1}) ={}& 1 \otimes x_{2kn-1} + t\otimes y_{2kn-1} + dt\otimes z_{2kn-2}.
    \end{align*}
    Since $da_{2k} = 0$ and the differential on the target vanishes, we deduce that $y_{2k} = 0$.
    The algebra homomorphism $H$ is therefore a dga-homomorphism if and only if
    \begin{align*}
           &dH(b_{2kn-1}) = dt \otimes y_{2kn-1}\\
        ={}&  H(db_{2kn-1}) =  (1 \otimes x_{2k} + dt \otimes z_{2k})^n = \sum_{\alpha=0}^n \binom{n}{\alpha} (dt)^{n-\alpha} \otimes x_{2k}^\alpha z_{2k}^{n-\alpha}= 1 \otimes x_{2k}^n + n \cdot dt \otimes x_{2k} ^{n-1} z_{2k}. 
    \end{align*}
    Since $\mathsf{A}$ is generated by odd elements, $x_{2k}^{n} \neq 0$ if and only if $n\geq 2$. 

    If $n\geq 3$, we deduce that $y_{2kn-1} = 0$.
    If $n=2$, then the additional and $0 = \mathrm{ev}_0 \circ H (a_{2k}) = x_{2k}$ also forces $y_{2kn-1}$ to vanish. 
    We conclude that $\mathrm{ev}_0 \circ H = \mathrm{ev}_1 \circ H$, as claimed.
\end{proof}
\begin{lemma}\label{lem: Sphere projection}
    The map $\mathrm{ev}_{e_1} \colon \SOrth(n+1) \rightarrow S^n$ given by evaluation at $e_1 = (1,0,\dots,0)$ has the following rational models:
    \begin{itemize}
        \item[(i)] If $n = 2k$, then $\mathsf{M}(\ev_{e_1})\colon \mathsf{M}_{S^n} \rightarrow \mathsf{M}_{\SOrth(n+1)}$ is induced by $a_n \mapsto 0$ and $b_{2n-1} \mapsto \pi_{4k-1}$.
        \item[(ii)] If $n = 2k+1$, then $\mathsf{M}(\ev_{e_1})\colon \mathsf{M}_{S^n} \rightarrow \mathsf{M}_{\SOrth(n+1)}$ is induced by $a_{2k+1} \mapsto \varepsilon_{2k+1}$.
    \end{itemize}
\end{lemma}
\begin{proof}
    The odd case is a well-known statement about rational cohomology of Lie groups, so we will prove the even-dimensional case only.
    The dga-homomorphism $\mathsf{M}(\ev_{e_1})$ modelling the evaluation map must be induced by
    \begin{equation*}
        a_{2k} \mapsto \mathrm{dec}_{2k} \qquad \text{ and } \qquad b_{4k-1} \mapsto \lambda_{4k-1}\pi_{4k-1} + \mathrm{dec}_{4k-1}
    \end{equation*}
    with $\lambda_{4k-1} \in \Q$ and decomposable elements $\mathrm{dec}_{2k} \in H^{2k}(\SOrth(2k+1)) = H^{2k}(\SOrth(2k))$ and $\mathrm{dec}_{4k-1} \in H^{4k-1}(\SOrth(2k))$. 
    Since the composition $\SOrth(2k) \hookrightarrow \SOrth(2k+1) \xrightarrow{\mathrm{ev}_1} S^{2k}$ is the constant map, this forces $\mathrm{dec}_{2k}$ to vanish, otherwise the constant map would produce a non-zero homomorphism between cohomology groups.
    Since the composition  $\mathsf{M}_{S^{2k}} \rightarrow H(\SOrth(2k+1)) \rightarrow H(\SOrth(2k))$ is homotopic to the zero map, and satisfies $a_{2k} \mapsto 0$, Lemma \ref{lem: rigid homotopy classes} implies $\mathrm{dec}_{4k-1}$ to vanish.
    
   The induced long exact sequence of rational homotopy groups together with their algebraic description in terms of their minimal models, see Theorem \ref{thm: homotopy classification} and Example \ref{Exmpl: Minimal model of Lie groups}, implies that $\lambda_{4k-1}$ must be non-zero.
   By rescaling $\pi_{4k-1}$ if necessary, we may assume $\lambda_{4k-1} = 1$. 
\end{proof}
The next example concerns complex projective spaces.
\begin{lemma}\label{lem: CP porojection}
    The map $\mathrm{ev}_{[e_0]} \colon \U(n+1) \rightarrow \CP^n$ given by evaluating at $[e_0] = [1:0:\dots:0]$ is modelled by the dga-homomorphism
    \begin{equation*}
        \mathsf{M}(\ev_{[e_0]}) \colon \mathsf{M}_{\CP^n} \rightarrow \mathsf{M}_{\U(n+1)} \quad \text{ induced by } \quad a_2 \mapsto 0 \text{ and } b_{2n+1} \mapsto \gamma_{2n+1}.
    \end{equation*}
\end{lemma}
\begin{proof}
    By applying the long exact sequence of rational homotopy groups to the fibration $\U(1)\times \U(n) \rightarrow \U(n+1) \xrightarrow{\ev_{[e_0]}} \CP^n$, we deduce $\pi_{2n+1}(\ev_{[e_0]})_\Q$ must be an isomorphism.
    From the natural isomorphism $\pi_n(X)_\Q \cong \mathrm{Hom}(V^n_X,\Q)$ for each nilpotent space of finite $\Q$-type, where $V_X$ is the generating graded vector space of the minimal model of $X$, we deduce that $\ev_{[e_0]}$ is modelled by a homotopy class of a dga-homomorphism of form\footnote{After rescaling $\gamma_{2n+1}$ if necessary.}
    \begin{equation*}
        a_2 \mapsto 0 \qquad \text{ and } \qquad b_{2n+1} \mapsto \gamma_{2n+1} + \mathrm{dec}_{2n+1},
    \end{equation*}
    where $\mathrm{dec}_{2n+1} \in \Lambda[\gamma_1,\dots, \gamma_{2n-1}]$ is a linear combination of decomposable elements.

    We prove by induction that $\mathrm{dec}_{2n+1} = 0$.
    Starting with $n=1$, we see $\mathrm{dec}_{3} = 0$ for algebraic reasons.

    Evaluation at $[e_0]$ is stable in the sense that the left square in following diagram commutes strictly, so right square commutes up to homotopy
    \begin{equation*}
        \xymatrix{  \U(n+1) \ar[rr]^{\ev_{[e_0]}} && \CP^n & \mathsf{M}_{\U(n+1)} \ar@{->>}[d]_{\mathsf{M}(\incl)} && \ar[ll]_{\mathsf{M}(\ev_{[e_0]})} \mathsf{M}_{\CP^n} \ar[d]^{b_{2n+1}\mapsto a_2b_{2n-1}} \\
        \U(n) \ar[u]^{\incl} \ar[rr]^{\ev_{[e_0]}} && \CP^{n-1} \ar[u] & \mathsf{M}_{\U(n)} && \ar[ll]_{\mathsf{M}(\ev_{[e_0]})} \mathsf{M}_{\CP^{n-1}}  }
    \end{equation*}
    where $\mathsf{M}(\incl) $ is given by $\gamma_j \mapsto \gamma_j$ and $\gamma_{2n+1} \mapsto 0$.
    By Lemma \ref{lem: rigid homotopy classes} the right diagram must commute, so we deduce $0=\mathsf{M}(\incl) \circ \mathsf{M}(\ev_{[e_0]}) = \mathsf{M}(\incl)(\gamma_{2n+1} + \mathrm{dec}_{2n+1}) = \mathrm{dec}_{2n+1}$.
\end{proof}
The argument of the previous proof generalises to the quaternionic projective spaces and the Cayley plane.
\begin{lemma}\label{lem: HP and OP projection}
    The evaluation maps $\mathrm{ev}_{[e_0]} \colon \Sp(n+1) \rightarrow \HP^n$ and $\mathrm{ev}_{[e_0]} \colon \Ffour \rightarrow \OP^2 = \Ffour/\Spin(9)$ are modelled by the dga-homomorphisms
    \begin{align*}
        \mathsf{M}_{\HP^n} \rightarrow \mathsf{M}_{\Sp(n+1)} \qquad &\text{ induced by } \qquad a_4 \mapsto 0 \ \text{ and } \  b_{4n+3} \mapsto q_{4n+3},\\
        \mathsf{M}_{\OP^2} \rightarrow \mathsf{M}_{\Ffour} \hspace{23pt}\qquad &\text{ induced by } \qquad a_8 \mapsto 0 \ \text{ and } \ b_{23} \mapsto \xi_{23}.        
    \end{align*}
\end{lemma}
Finally, we would like to calculate the rational homotopy groups of the mapping spaces $\mathrm{C}(\KP^m,\KP^n)$ and $\mathrm{C}(S^m,S^n)$ using  Theorem \ref{Thm: Mapping Spaces via derivations}.
We begin with discussing the case of projective spaces.
%
\begin{lemma}\label{lem: h-groups mapping spaces}
    For $1 \leq m < n$, $k=\mathrm{dim}_\C(\mathbb{K})$, and the standard inclusion $\incl \colon \KP^m \hookrightarrow \KP^n$, the rational homotopy groups of the mapping spaces $\mathrm{C}(\KP^m,\KP^n)$ are given by
    \begin{equation*}
        \pi_r( \mathrm{C}(\KP^m,\KP^n),\incl)_\Q \cong  \begin{cases}
            \Q, &\text{if } r \in \{ 2k(n-\alpha) + 2k+1 \, : \, 0 \leq \alpha \leq m \} \cup \{2k\},  \\
            0, &\text{else}.
        \end{cases}
    \end{equation*}
    If $m=n$, we have 
    \begin{equation*}
        \pi_r( \mathrm{hAut}(\mathbb{K}P^n),\id)_\Q \cong  \begin{cases}
            \Q, &\text{if } r \in \{ 2k(n-\alpha) + 2k+1 \, : \, 0 \leq \alpha \leq n \},  \\
            0, &\text{else}.
        \end{cases}
    \end{equation*}
\end{lemma}
\begin{proof}
    Recall that the minimal models of $\KP^n$ is given by $\Lambda[a_{2k},b_{2k(n+1)-1} \, | \, db_{2k(n+1)-1} = a_{2k}^{n+1}]$ and that the inclusion is modelled by the dga-homomorphism $\mathsf{M}(\incl)$ that sends $a_{2k}$ to $a_{2k}$ and $b_{2k(n+1)-1}$ to $a_{2k}^{n-m}b_{2k(m+1)-1}$.

    Since derivations with domain a free algebra are uniquely determined by the images of the basis of the generating vector space, we deduce that the vector space $\Der(\mathsf{M}_{\KP^n},\mathsf{M}_{\KP^m})$ of all positive graded $\mathsf{M}(\incl)$-derivation is generated by the derivations
    \begin{equation*}
        1 \otimes a_{2k}^\vee, \qquad a^\alpha_{2k} \otimes b_{2k(n+1)-1}^\vee, \qquad a_{2k}^\beta b_{2k(m+1)-1}\otimes b_{2k(n+1)-1}^\vee
    \end{equation*}
    with $0 \leq \alpha \leq n$ and $0 \leq \beta < n-m$.
    The derivations are of degree $2k$, $2k(n-\alpha)+2k-1$, and $2k(n-m)-2k\beta$, respectively.

    Straightforward computations yield the identities
    \begin{align*}
        \delta(1 \otimes a_{2k}^\vee) &= -(n+1)a_{2k}^n\otimes b_{2k(n+1)-1}^\vee, \qquad \delta(a_{2k}^\alpha \otimes b_{2k(n+1)-1}^\vee) = 0, \\
        \delta(a_{2k}^\beta b_{2k(m+1)-1}\otimes b_{2k(n+1)-1}^\vee) &= a_{2k}^{\beta+m+1}\otimes b_{2k(n+1)-1}^\vee. 
    \end{align*}
    We conclude that $\ker \delta$ is spanned by the derivations $a_{2k}^\alpha \otimes b_{2k(n+1)-1}$ and  $1 \otimes a_{2k}^\vee + (n+1)a_{2k}^{n-m}b_{2k(m+1)-1} \otimes b_{2k(n+1)-1}^\vee$ if $m<n$, and spanned by $a_{2k}^\alpha \otimes b_{2k(n+1)-1}$ if $m=n$. 
    The image of $\delta$ is generated by the derivations $a_{2k}^\alpha \otimes b_{2k(n+1)-1}$ with $m < \alpha \leq n$ if $m<n$ and by $a_{2k}^n \otimes b_{2k(n+1)-1}$ if $m=n$.
    The resulting homology groups are then generated by $a_{2k}^\alpha \otimes b_{2k(n+1)-1}^\vee$ and $1 \otimes a_{2k}^\vee + (n+1)a_{2k}^{n-m}b_{2k(m+1)-1} \otimes b_{2k(n+1)-1}^\vee$ for $0 \leq \alpha \leq m<n$ and by $a_{2k}^n \otimes b_{2k(n+1)-1}^\vee$, so we deduce
    \begin{equation*}
        H_r(\Der(\mathsf{M}_{\KP^n},\mathsf{M}_{\KP^m}),\delta) = \begin{cases}
            \Q, &\text{if } r \in \{ 2k(n-\alpha) + 2k+1 \, : \, 0 \leq \alpha \leq m \} \cup \{2k\} \\
            0, &\text{else}.
        \end{cases}
    \end{equation*}
    and
    \begin{equation*}
        H_r(\Der( \mathrm{hAut}(\mathbb{K}P^n),\delta) = \begin{cases}
            \Q, &\text{if } r \in \{ 2k(n-\alpha) + 2k+1 \, : \, 0 \leq \alpha \leq n-1 \}  \\
            0, &\text{else}.
        \end{cases}
    \end{equation*}
    The claim now follows from Theorem \ref{Thm: Mapping Spaces via derivations}.
\end{proof}
The corresponding result for spheres can be obtained in the same fashion.
However, if $m<n$ we have the following alternative, because, in this case, the inclusion $S^m \hookrightarrow S^n$ is nullhomotopic.
Thus, the fibration $\mathrm{C}_\ast(S^m,S^n) \rightarrow \mathrm{C}(S^m,S^n) \xrightarrow{\ev_{e_1}} S^n$ with the fibre the space $\mathrm{C}_\ast(S^m,S^n)$ of all base-point preserving continuous maps $S^m \rightarrow S^n$ has a section that sends a point $x$ to the constant map $\mathrm{const}_x \colon S^m \rightarrow S^n$ with value $x$.
From this, we deduce 
\begin{equation*}
    \pi_k(\mathrm{C}(S^m,S^n))_\Q = \pi_k(S^n)_\Q \oplus \pi_k(\Omega^mS^n)_\Q = \pi_k(S^n)_\Q \oplus \pi_{k+m}(S^n)_\Q,
\end{equation*}
which allows us to read off the rational homotopy groups from the minimal models of spheres.
We conclude the next result.
\begin{lemma}\label{lem: h-groups mapping spaces spheres}
    For $1 \leq m \leq n$, the rational homotopy groups $\pi_k(\mathrm{C}(S^m,S^n),\incl)_\Q$ are given as follows:
    \begin{itemize}
        \item[(i)] If $n$ is odd, then $\pi_k(\mathrm{C}(S^m,S^n),\incl)_\Q = \Q$ if $k=n$ or $k=n-m>0$ and zero otherwise. 
        \item[(ii)] If $n$ is even and $m<n$, then $\pi_\ast( \mathrm{C}(S^m,S^n),\incl)_\Q$ is a four dimensional vector space with generators in degree $n-m,n,2n-1,$ and $2n-m-1$.  
        \item[(iii)] If $n=m$ even, then $\pi_k(\hAut(S^n),\id)_\Q = \Q$ if $k = 2n-1$.
    \end{itemize}
\end{lemma}
\section{Robustness of Symmetries}\label{Section: Robustness} 
%
%
%
We now prove the main results of this article by calculating rational models for the actions $\ell \colon \mathrm{Sym}(M,g) \times M \rightarrow M$ for the simply connected symmetric spaces listed in Table \ref{table: table of symmetric spaces.}.
\begin{table}[ht]
\centering
\begin{tabular}{|m{2.6cm}||m{2.6cm}|m{2.6cm}|m{2.6cm}|m{2.6cm}|}
\hline
 $M$ & $S^n$ & $\CP^n$ & $\HP^n$ & $\OP^2$ \\ \hline
$\mathrm{Sym}(M)$   &  $\SOrth(n+1)$     &   $\U(n+1)$    &  $\Sp(n+1)$     & $\Ffour$    \\ \hline
$\mathrm{Iso}(M)_0$  &  $\SOrth(n+1)$    &   $\U(n+1)/\U(1)$     &  $\Sp(n+1)/\Z_2$     & $\Ffour$     \\ \hline
\end{tabular}\vspace{3pt}
\caption{List of symmetric spaces we consider with symmetry group and the connected component of the unit in the associated isometry group.}
\label{table: table of symmetric spaces.}
\end{table}

The basis for the proofs will be the following technical lemma, which should be considered as a generalisation of Lemma \ref{lem: rigid homotopy classes}.
To formulate it, recall that a cdga $\mathsf{A}$ is \emph{connected} if $\mathsf{A}^0 = \Q\cdot 1$ and that it is \emph{augmented} if there exists a dga-homomorphism $\varepsilon_{\mathsf{A}} \colon \mathsf{A}\rightarrow \Q$.
Recall further that an arbitrary homotopy $H \colon \Lambda V \rightarrow \Lambda[t,dt] \otimes \mathsf{A} \otimes \Lambda W$ with $\mathsf{A}$ an augmented cdga gives rise to the following commutative diagram:
    \begin{equation*}
        \xymatrix{ && && \Q \otimes \Lambda W  \\
        \Lambda V \ar[rr]^-H \ar@<0.3pc>@/^1pc/[rrrru]^{\varphi_0}_{\varphi_1} \ar@<-0.3pc>@/_1pc/[rrrrd]^{\psi_0}_{\psi_1} && \Lambda[t,dt] \otimes \mathsf{A} \otimes \Lambda W \ar@<0.3pc>[rr]^{\mathrm{ev}_0} \ar@<-0.3pc>[rr]_{\mathrm{ev}_1} && \mathsf{A} \otimes \Lambda W  \ar@{->>}[u]_{\otimes \id}^{ \varepsilon_\mathsf{A}} \ar@{->>}[d]^{\otimes\varepsilon_{\Lambda W}}_{\id} \\
        && && \mathsf{A} \otimes \Q. }
    \end{equation*}
\begin{prop}\label{prop: description of homotopies}
    For $k\geq1$ and $n\geq m \geq 2$, let $(\Lambda V,d) = \Lambda[a_{2k},b_{{2kn-1}} \, |\, db_{2kn-1} = a_{2k}^n]$ and $(\Lambda W,d) = \Lambda[ a_{2k} , b_{2mk-1} \, | \, db_{2mk-1} = a_{2k}^m ]$ be two Sullivan algebras and let $(\mathsf{A},0)$ be an augmented, connected cdga with trivial differential.
    Then each homotopy $H \colon (\Lambda V,d) \rightarrow \Lambda[t,dt] \otimes \mathsf{A} \otimes (\Lambda W,d)$ is of the form
    \begin{align*}
        H(a_{2k}) &= 1 \otimes 1 \otimes \lambda_0\cdot a_{2k} + 1 \otimes \psi_0(a_{2k}) \otimes 1  + dt \otimes x_{2k-1} \otimes 1,  \\
        \begin{split}
            H(b_{2nk-1}) ={}& 1 \otimes \Bigl( \lambda_0^{n-m} \cdot 1\otimes a_{2k}^{n-m} b_{2km-1} +  \psi_0(b_{2kn-1}) \otimes 1 + h_0(b_{2kn-1}) \Bigr) \\
            &+t \otimes \left( n\cdot \psi_0(a_{2k})^{n-1} x_{2k-1}  + n \cdot \sum_{\alpha=1}^{n-1} \psi_0(a_{2k})^{n-\alpha-1}x_{2k-1} \otimes (\lambda_0a_{2k})^\alpha + dh_{dt}(b_{2kn-1})\right) \\
            &+dt \otimes h_{dt}(b_{2kn-1}),
        \end{split}
    \end{align*}
    where\footnote{We implicitly assume that $H$ is a dga-homomorphism, of course, so that $d\psi(b_{2kn-1}) = \psi_0(a_{2k})^n =0$.} $\lambda_0 \in \Q$, $x_{2k-1} \in \mathsf{A}^{2k-1}$ is closed, $\psi_0(a_{2k})^{r} = 0$ for $r > n-m$, 
    and $h_0(b_{2kn-1})\in \mathsf{A}^{\geq1} \otimes \Lambda^{\geq 1} W$ is a primitive of $\sum_{\alpha=m}^{n-1} \binom{n}{\alpha} \psi_0(a_{2k})^{n-\alpha} \otimes (\lambda_0a_{2k})^{\alpha}  ) $, while $h_{dt}(b_{2kn-1}) \in \mathsf{A} \otimes \Lambda W$ satisfies $dh_{dt}(b_{2kn-1}) \in \mathsf{A} \otimes a_{2k}^m \cdot \Lambda W$ \footnote{Note that if $m=n$, then this condition implies that $dh_{dt}(b_{2kn-1}) =0$.}. 
\end{prop}
We mostly apply this proposition in the case where $\psi_0(a_{2k}) = 0$, so let us  spell out the simplified formula explicitly.
\begin{cor}\label{cor: simplified description of homotopies}
    If, in addition, $\psi_0(a_{2k})=0$ in Proposition \ref{prop: description of homotopies}, then all possible homotopies are of the form   
    \begin{align*}
        H(a_{2k}) ={}& 1 \otimes 1 \otimes \lambda_0 \cdot a_{2k}  + dt \otimes x_{2k-1} \otimes 1, \\
        \begin{split}
            H(b_{2kn-1}) ={}& 1\otimes \bigl( \lambda_0^{n-m} \cdot 1 \otimes a_{2k}^{n-m}b_{2km-1} + \psi_0(b_{2kn-1}) \otimes 1 +  h_0(b_{2kn-1}) \bigr) + \\
            &t \otimes \bigl(n \lambda_0^{n-1}  x_{2k-1} \otimes a_{2k}^{n-1} + dh_{dt}(b_{2kn-1})\bigr) + dt\otimes h_{dt}(b_{2kn-1})
        \end{split}
    \end{align*}
    with $x_{2k-1} \in \mathsf{A} ^{2k-1}$ closed, $h_{dt}(b_{2kn-1}) \in (\mathsf{A} \otimes \Lambda W  )^{2kn-2}$ an element with $dh_{dt}(b_{2kn-1}) \in (\mathsf{A} \otimes a_{2k}^{m}\cdot \Lambda W)^{2k(n-m)-1}$ and $h_0(b_{2kn-1}) \in \mathsf{A}^{\geq 1} \otimes \Lambda^{\geq 1} W$ closed.
\end{cor}
\begin{rem}
    Informally speaking, the degree of freedom we have to change a dga-homomorphism $\Phi_0 = \mathrm{ev}_0 \circ H$ within its homotopy class is the coefficients behind the $t$-factor.
    Proposition \ref{prop: description of homotopies} tells us that homotopies are quite rigid, because they can only perturb $\Phi_0$ by an element $x_{2k-1} \in \mathsf{A}^{2k-1}$ and an exact element $dh_{dt}(b_{2kn-1}) \in  ( \mathsf{A} \otimes a_{2k}^m \cdot \Lambda W)^{2k(n-m)-1}$ (which is automatically zero if $n=m$).
\end{rem}
\begin{proof}
    We can uniquely decompose $H(a_{2k})$ and $H({b_{2kn-1}})$ as follows:
    \begin{align*}
         \begin{split}
            H(a_{2k}) ={}& 1 \otimes \bigl( 1 \otimes \varphi_0(a_{2k}) + \psi_0(a_{2k}) \otimes 1 \bigr) + \\
            &t \otimes \bigl( 1 \otimes (\varphi_1 - \varphi_0)(a_{2k}) +  (\psi_1 - \psi_0)(a_{2k}) \otimes 1 \bigr) + dt \otimes x_{2k-1} \otimes 1
        \end{split}  \\
        \begin{split}
            H(b_{2kn-1}) ={}& 1 \otimes \bigl(  1 \otimes \varphi_0(b_{2kn-1})  + \psi_0(b_{2kn-1}) \otimes 1 + h_0(b_{2kn-1}) \bigr) + \\
            & t \otimes \bigl( 1 \otimes (\varphi_1 - \varphi_0)(b_{2kn-1}) + (\psi_1 - \psi_0)(b_{2kn-1}) \otimes 1 + h_t(b_{2kn-1}) \bigr) + \\
            & dt \otimes h_{dt}(b_{2kn-1}),
        \end{split}
    \end{align*}
    for some elements $x_{2k-1} \in \mathsf{A}^{2k-1}$, $h_0(b_{2kn-1}), h_t(b_{2kn-1}) \in \mathsf{A}^{\geq 1} \otimes \Lambda^{\geq 1} W$, and $h_{dt}(b_{2kn-1}) \in (\mathsf{A} \otimes \Lambda W)^{2kn-2}$.
    Conversely, every such assignment determines a unique algebra homomorphism $H \colon \Lambda V \rightarrow I \otimes \mathsf{A} \otimes \Lambda W$ and $H$ is a dga-homomorphism if and only if $H(da_{2k}) = dH(a_{2k})$ as well as $H(db_{2kn-1}) = dH(b_{2kn-1})$ are satisfied.

    The first equation translates into 
    \begin{equation*}
        0 = H(da_{2k}) = dH(a_{2k}) = dt \otimes \bigl( 1 \otimes (\varphi_1 - \varphi_0)(a_{2k}) + (\psi_1 - \psi_0)(a_{2k}) \otimes 1 + dx_{2k-1} \otimes 1\bigr),
    \end{equation*}
    which is equivalent to 
    \begin{equation}\label{eq: homotopy rigidity on a}
        \varphi_0(a_{2k}) = \varphi_1(a_{2k}) = \lambda_0 a_{2k}, \qquad \psi_0(a_{2k}) = \psi_1(a_{2k}), \qquad \text{ and } \qquad  dx_{2k-1} =0. 
    \end{equation}
    for some $\lambda_0 \in \Q$.
    
    The left hand side of the second identity is now given by
    \begin{equation}
       \begin{split}\label{eq: homotopy regidity 2k}
           H(db_{2kn-1}) = H(a_{2k})^n &= \sum_{\alpha=0}^n \binom{n}{\alpha}\cdot 1 \otimes \psi_0(a_{2k})^{n-\alpha} \otimes \varphi_0(a_{2k})^\alpha \\
           &+ n\sum_{\alpha=0}^{n-1} \binom{n-1}{\alpha} dt  \otimes \psi_0(a_{2k})^{n-\alpha-1} x_{2k-1} \otimes \varphi_0(a_{2k})^\alpha,
       \end{split}   
    \end{equation}
    while the right hand side is given by
    \begin{equation*}
        \begin{split}
            dH(b_{2kn-1}) ={} & \  \ \ 1 \otimes \bigl(1 \otimes d\varphi_0(b_{2kn-1}) + d\psi_0(b_{2kn-1}) \otimes 1 + dh_0(b_{2kn-1})\bigr) \\
            & + t \otimes \bigl( 1 \otimes d(\varphi_1 - \varphi_0)(b_{2kn-1})  +  d(\psi_1 - \psi_0)(b_{2kn-1}) \otimes 1 + dh_t(b_{2kn-1}) \bigr) \\
            & + dt \otimes \bigl( 1 \otimes (\varphi_1 - \varphi_0)(b_{2kn-1}) + (\psi_1 - \psi_0)(b_{2kn-1}) \otimes 1 +  h_t(b_{2kn-1}) - dh_{dt}(b_{2kn-1}) \bigr)
        \end{split}
    \end{equation*}
    Degree reasons imply that $h_{dt}(b_{2kn-1}) \in (\mathsf{A} \otimes \Lambda W )^{2nk-2} = (\mathsf{A} \otimes \Lambda[a_{2k}, b_{2km-1}])^{2kn-2}$, so that $dh_{dt}(b_{2kn-1}) \in (\mathsf{A} \otimes a_{2k}^m \cdot \Lambda W)$.
    In particular, if $n=m$, then $(\Lambda W \otimes \mathsf{A})^{2nk-2} = (\Lambda[a_{2k}] \otimes \mathsf{A})^{2kn-2}$ and $h_{dt}(b_{2kn-1})$ must be closed.
    By `comparing coefficients' the derive that the identity $dH(b_{2kn-1}) = Hdb_{2kn-1})$ is satisfied if and only if the following identities hold true: 
    \begin{align}
        \varphi_0(a_{2k})^n ={}& d \varphi_0(b_{2kn-1}), \quad \text{ and } \quad  d \psi_0(b_{2kn-1}) = \psi_0(a_{2k})^n, \label{eq: chain alg requirement}\\
        dh_0(b_{2kn-1}) ={}& \sum_{\alpha=1}^{n-1} \binom{n}{\alpha} \psi_0(a_{2k})^{n-\alpha} \otimes \varphi_0(a_{2k})^\alpha \label{eq: exactness of product}\\
        \varphi_1(b_{2kn-1}) ={}& \varphi_0(b_{2kn-1}) \qquad \text{ and } \qquad \psi_1(b_{2kn-1}) = \psi_0(b_{2kn-1}) + n\cdot \psi_0(a_{2k})^{n-1}x_{2k-1},  \\
        h_t(b_{2kn-1}) ={}& n\sum_{\alpha=1}^{n-1} \binom{n-1}{\alpha}  \psi_0(a_{2k})^{n-\alpha-1} \cdot x_{2k-1} \otimes \varphi_0(a_{2k})^{\alpha} + dh_{dt}(b_{2kn-1}). \label{eq: homotopy rigidity final}
    \end{align}
    Note that the first condition of (\ref{eq: chain alg requirement}) implies that $\varphi_0(b_{2kn-1}) = \varphi_0(a_{2k})^{n-m}b_{2km-1}$.
    Furthermore, observe that (\ref{eq: exactness of product}) is a condition: it forces the right-hand side to be exact, and this holds true if and only if either $\varphi_0(a_{2k}) = 0$ or if $\psi_0(a_{2k})^{n-\alpha} = 0$ for all $\alpha < m$ because $a_{2k}^m = db_{2km-1}$ is an exact element in $\Lambda W$ and $\mathsf{A}$ is equipped with the zero differential.

    Plugging (\ref{eq: homotopy rigidity on a}) and (\ref{eq: chain alg requirement}) - (\ref{eq: homotopy rigidity final}) and the just derived observations into the defining equation of $H(a_{2k})$ and $H(b_{2kn-1})$ yields the claimed formulas.  
\end{proof}
\begin{cor}\label{cor: strict commuativity}
    If $(M,g)$ is a simply connected symmetric space of rank $1$ and $\mathrm{Sym}(M,g)$ its symmetry group that is considered in Example \ref{Exmpl: Minimal model of Lie groups}, then the rational model $\mathsf{M}(\ell)$ of its symmetric action $\ell$ makes the following diagram commutative
    \begin{equation*}
        \xymatrix{ && \Q \otimes \mathsf{M}_{M} \\
        \mathsf{M}_{M} \ar[rr]^-{\mathsf{M}(\ell)} \ar@/^1pc/[rru]^{\id} \ar@/_1pc/[rrd]_{\mathsf{M}(\mathrm{ev}_{\ast})}&& H(\mathrm{Sym}(M,g)) \otimes \mathsf{M}_{M} \ar@{->>}[u]_{\otimes \id}^{\varepsilon} \ar@{->>}[d]^{\varepsilon}_{\id \otimes } \\ 
        && H(\mathrm{Sym}(M,g)) \otimes \Q. }
    \end{equation*}
\end{cor}
\begin{proof}
    Since $\ell$ restricts to $\mathrm{ev}_{\ast} = \ell(\placeholder,\ast)$ on $\mathrm{Sym}(M,g) \times \{\ast\}$ and to the identity on $\{1\} \times M$, we know that the diagram in the statement must commute up to homotopy.
    Since the augmentation maps are surjective, they satisfy the the homotopy lifting property \cite[Proposition 5.3]{Bousfield1976PLdeRham}.
    After changing $\mathsf{M}(\ell)$ in its homotopy class if necessary, we assume that the lower triangle commutes strictly.
    
    Since all rank 1-symmetric spaces have a minimal model of the form $\Lambda[a_{2k},b_{2kn-1} \, | \, db_{2kn-1} = a_{2k}^n]$, we can apply Proposition \ref{prop: description of homotopies} with $\mathsf{A} = \Q$ and $n=m$ to deduce that every dga-homomorphism is unique in its homotopy class (because in the notation of Proposition \ref{prop: description of homotopies} $x_{2k-1} = 0$ for degree reasons). 
    Thus, the upper triangle strictly commutes, too.
\end{proof}
We begin with the description of the rational model for left action $\ell \colon \SOrth(n+1) \times S^m \rightarrow S^n$.
\begin{prop}\label{prop: left action model spheres}
    The rational model is given by the homotopy class of the dga-homomorphism
    \begin{align*}
        a_{2k} &\mapsto 1 \otimes a_{2k}, \qquad b_{4k-1} \mapsto 1 \otimes b_{4k-1} + \pi_{4k-1} \otimes 1, & \text{ if } n=2k,\mathcolor{white}{+1} \\
        a_{2k+1} &\mapsto 1 \otimes a_{2k+1} + \varepsilon_{2k+1} \otimes 1, & \text{ if } n=2k+1. \\
    \end{align*}
\end{prop}
\begin{proof}
   For $n=m=2k$ even, Lemma \ref{lem: Sphere projection} and Corollary \ref{cor: strict commuativity} imply that a dga-homomorphism $\varphi_0$ modelling $\ell$ must send $a_{2k}$ to $1 \otimes a_{2k}$ and $b_{4k-1} \mapsto 1 \otimes b_{4k-1} + \pi_{4k-1}\otimes 1 + \mathrm{dec}_{2k-1} \otimes a_{2k-1}$ with $\mathrm{dec}_{2k-1} \in H^{2k-1}(\SOrth(2k+1))$.
   By applying Corollary \ref{cor: simplified description of homotopies} with $\mathsf{A}= H(\SOrth(2k+1)) $, $x_{2k-1} = -\mathrm{dec}_{2k-1}/2$ and $h_{dt}(b_{4k-1}) =0$ we find a homotopy between $\varphi_0$ and $\varphi_1 = \mathsf{M}(\ell)$ in the claim.

   In the odd case, each rational model $\varphi_0$ for $\ell$ must be of the form $\varphi(a_{2k+1}) = 1 \otimes a_{2k+1} + \varepsilon_{2k+1} \otimes 1 + \mathrm{dec}_{2k+1} \otimes 1$ with $\mathrm{dec}_{2k+1} \in H^{2k+1}(\SOrth(2k+1))$.
   Corollary \ref{cor: strict commuativity} together with Lemma \ref{lem: Sphere projection} implies now that $\mathrm{dec}_{2k+1} \otimes 1 = 0$.
\end{proof}
\begin{rem}\label{rem: reduction to evaluation}
    Since $S^m \rightarrow S^n$ is homotopic to the the constant map if $m<n$, the left action $\mathrm{SO}(n+1) \times S^m \rightarrow S^n$ is homotopic to $\mathrm{ev}_1 \colon \mathrm{SO}(n+1) \rightarrow S^n$, which was discussed in Lemma \ref{lem: Sphere projection}.
\end{rem}
Next we have a look at the left action $\ell \colon \mathrm{Sym}(\mathbb{K}P^n,g_{FS}) \times \mathbb{K}P^m \rightarrow \mathbb{K}P^n$.
\begin{prop}\label{prop: left action model projectives spaces}
    The left action $\ell$ is modelled by the (homotopy class) of the  dga-homomorphisms $\mathsf{M}(\ell)\colon \mathsf{M}_{\mathbb{K}P^n} \rightarrow H(\mathrm{Sym}(\KP^n,g_{FS})) \otimes \mathsf{M}_{\mathbb{K}P^n}$ induced by\footnote{We use the convention that $\xi_7=0$.}
    \begin{align*}
        \mathsf{M}(\ell)(a) &= 1 \otimes a, & \text{for } \mathbb{K} \in \{\C,\Quat,\mathbb{O}\},\\
        \mathsf{M}(\ell)(b_{2n+1}) &= 1 \otimes a_{2}^{n-m}b_{2n+1} + \sum_{\alpha=n-m}^n \gamma_{2\alpha + 1} \otimes a_{2}^{n-\alpha}, & \text{for } \mathbb{K} = \mathbb{C}, \\
        \mathsf{M}(\ell)(b_{4n+3}) &= 1 \otimes a_{4}^{n-m}b_{4n+3} + \sum_{\alpha=n-m}^n q_{4\alpha + 3} \otimes a_{4}^{n-\alpha}, & \text{for } \mathbb{K} = \mathbb{\Quat},\\
        \mathsf{M}(\ell)(b_{8n+7}) &= 1 \otimes a_{8}^{n-m}b_{8n+7} + \sum_{\alpha=n-m}^n \xi_{8\alpha + 7} \otimes a_{8}^{n-\alpha}, & \text{for } \mathbb{K} = \mathbb{\mathbb{O}}.
    \end{align*}
    Moreover, the coefficients are homotopy‑invariant in the sense that they remain unchanged if we replace $\mathsf{M}(\ell)$ by any homotopic representative.
\end{prop}
\begin{proof}
    We spell out the proof for $\HP^n$ only as the proofs for the other cases only differ in notation.
    We will prove the statement by induction over $n$ for the special case $n=m$ and then deduce the more general case $m\leq n$.

    We begin with $n=1$.
    For degree reasons and Lemma \ref{lem: CP porojection}, the dga-homomorphism $\mathsf{M}(\ell)$ must be of the form $a_4 \mapsto 1 \otimes a_4$ and $b_7 \mapsto 1 \otimes b_{7} + q_7 \otimes 1 + \mu\cdot q_3 \otimes a_4$ for some $\mu \in \Q$.

    Since $\mathrm{deg}(q_3) = \mathrm{deg}(a_4)-1$, we may choose $x_{3} = -\mu/2 \cdot  q_3 \otimes 1$ and $h_{dt} = 0$ in Proposition \ref{prop: description of homotopies} to change $\mathsf{M}(\ell)$ in its homotopy classes to set $\mu=0$, which is not possible for the other coefficients. 

    For $n>1$, we argue inductively.
    As for $\HP^1$, we know that $\mathsf{M}(\ell)$ must have the form
    \begin{equation*}
        a_{4n} \mapsto 1 \otimes a_{4n} \qquad \text{ and } \qquad b_{4n+3} \mapsto 1 \otimes b_{4n+3} + \sum_{\alpha = 0}^n {}_n\zeta_\alpha \otimes a_{4}^{n-\alpha},
    \end{equation*}
    with ${}_n\zeta_\alpha \in H^{4\alpha + 3}(\mathrm{Sym}(\HP^n))$.
    By Corollary \ref{cor: simplified description of homotopies}, we see that all coefficients except ${}_n\zeta_0$ are homotopy invariant. 
    As in the special case $n=1$, we can set ${}_n\zeta_0=0$ if we wish.

    The top coefficient ${}_n\zeta_n \in H^{4n+3}(\mathrm{Sym}(\HP^n))$ can be determined using Corollary \ref{cor: strict commuativity} and Lemma
    Lemma \ref{lem: HP and OP projection}: It satisfies ${}_n\zeta_n = q_{4n+3}$.
    
    To determine the coefficients, we consider the following diagram that commutes up to homotopy
    \begin{equation*}
        \xymatrix@R-.7em{ \mathsf{M}_{\HP^n} \ar[rr] \ar[d]_{\mathsf{M}(\incl)} && H(\mathrm{Sym}(\HP^n,g_{FS})) \otimes \mathsf{M}_{\HP^n} \ar[d]_{H(\iota) \otimes}^{\mathsf{M}(\incl)} \\
        \mathsf{M}_{\HP^{n-1}} \ar[rr] && H(\mathrm{Sym}(\HP^{n-1},g_{FS})) \otimes \mathsf{M}_{\HP^{n-1}}. }
    \end{equation*}
    Applying Corollary \ref{cor: simplified description of homotopies} with the cdgas $\Lambda W = \mathsf{M}_{\HP^{n-1}}$ and $\mathsf{A} = H(\mathrm{Sym}(\HP^{n-1},g_{FS}))$ and $\psi_0(a_4)=0$, we see that coefficients ${}_{n-1}\zeta_\alpha$ remain homotopy-invariant except for ${}_{n-1}\zeta_0$.
    From $\mathsf{M}(\ell) \circ \mathsf{M}(\incl) \simeq H(\iota) \otimes \mathsf{M}(\incl) \circ \mathsf{M}(\ell)$ and the homotopy invariance of the coefficients ${}_n\zeta_\alpha$ and ${}_{n-1}\zeta_\alpha$ for ($\alpha \neq 0$) we deduce that
    \begin{equation*}
        \bigl(H(\iota) \otimes \mathsf{M}(\incl)\bigr) \circ \mathsf{M}(\ell)(b_{4n+3}) = 1 \otimes a_4b_{4n-1} + \sum_{\alpha=0}^{n-1} {}_n\zeta_\alpha \otimes a_4^{n-\alpha}
    \end{equation*}
    must agree (up to a summand of the form ${}_n\zeta_0 \otimes a_4^{n-1} = \mu \cdot q_3 \otimes a_4^{n-1}$) with
    \begin{equation*}
        \bigl(H(\iota) \otimes \mathsf{M}(\incl)\bigr) \circ \mathsf{M}(\ell)(b_{4n+3}) = 1 \otimes a_4b_{4n-1} + \sum_{\alpha=0}^{n-1} {}_{n-1}\zeta_\alpha \otimes a_4^{n-\alpha}.
    \end{equation*}
    By comparing coefficients we derive the recursive equation ${}_n \zeta_{\alpha} = {}_{n-1} \zeta_{\alpha}$ which determines all coefficients ${}_n\zeta_\alpha$ with $0 < \alpha <n$ to be ${}_n\zeta_\alpha = q_{4\alpha + 3}$, which proves the statement for the case $m=n$.

    \vspace{6pt}

    \noindent To obtain the general case, we use the homotopy‑commutative diagram
    \begin{equation*}
        \xymatrix@R-.7em{ \mathsf{M}_{\HP^n} \ar[rr]^-{\mathsf{M}(\ell)} \ar[rrd]_-{\mathsf{M}(\ell)} && H(\mathrm{Sym}(\HP^n,g_{FS})) \otimes \mathsf{M}_{\HP^n} \ar[d]_{\id \otimes}^{\mathsf{M}(\incl)} \\
        && H(\mathrm{Sym}(\HP^{n},g_{FS})) \otimes \mathsf{M}_{\HP^{m}}. }
    \end{equation*}
    In this case, $\mathsf{M}(\ell) \colon \mathsf{M}_{\HP^n} \rightarrow H(\mathrm{Sym}(\HP^n)) \otimes \mathsf{M}_{\HP^m}$ is induced by the assignment
    \begin{equation*}
        a_{4} \mapsto a_4 \quad \text{ and } \quad b_{4n+3} \mapsto 1 \otimes a_4^{n-m}b_{4n+3}+\sum_{\alpha = 0}^n q_{4\alpha +3} \otimes a_4^{n-\alpha}.
    \end{equation*}
    However, since $a_4^{n-\alpha}$ is exact for $\alpha<m$ we can choose the homotopy $H \colon \mathsf{M}_{\HP^n} \rightarrow \Lambda[t,dt] \otimes H(\mathrm{Sym}(\HP^n,g_{FS}) \otimes \mathsf{M}_{\HP^m}$ that is given by
    \begin{align*}
        H(a_4) ={}& 1\otimes 1\otimes a_4 \\
        \begin{split}
            H(b_{4n+3}) ={}& 1 \otimes \Bigl(1\otimes a_{4}^{n-m}b_{4m+3} + \sum_{\alpha=0}^n q_{4\alpha + 3} \otimes a_{4}^{n-\alpha} \Bigr) \\
            {}& - t \otimes \sum_{\alpha=0}^{n-(m+1)} q_{4\alpha + 3} \otimes a_{4}^{n-\alpha}  + dt\otimes \sum_{\alpha=0}^{n-(m+1)} q_{4\alpha + 3} \otimes a_{4}^{n-\alpha-(m+1)}b_{4m+3}, 
        \end{split}
    \end{align*}
    to obtain the claimed formula.
\end{proof}
%
We finally in the position to prove the main results of this article.
\begin{theorem}\label{thm: general main result rhg projective spaces}
    For all $1 \leq m \leq n$, the inclusions $s \colon \mathrm{Sym}(\KP^n,g_{FS})/\mathrm{Stab}(\KP^m) \rightarrow \mathrm{C}(\KP^m,\KP^n)$ induced by the left action induce isomorphisms on rational homotopy groups in the following degrees: 
    \begin{align*}
        &\pi_k(s) \colon \pi_k(\U(n+1)/(\U(1)\times \U(n-m))_\Q \xrightarrow{\cong} \pi_k(\mathrm{C}(\CP^m,\CP^n),\incl)_\Q & \text{for all } k > 0, \\
        &\pi_k(s) \colon \pi_k(\Sp(n+1)/(\Z_2\times \Sp(n-m))_\Q \xrightarrow{\cong} \pi_k(\mathrm{C}(\HP^m,\HP^n),\incl)_\Q & \text{for all } k>3, \\
        &\pi_k(s) \colon \pi_k(\Ffour)_\Q \xrightarrow{\cong} \pi_k(\hAut(\OP^2),\id)_\Q & \text{for all } k>11, \\
        &\pi_k(s) \colon \pi_k(\Ffour/\Spin(7))_\Q \xrightarrow{\cong} \pi_k(\mathrm{C}(\OP^1,\OP^2),\incl)_\Q & \text{for all } k>0.
    \end{align*}
    In the remaining cases, the target groups are zero and the kernels of the homomorphisms are given by $\pi_3(\mathrm{Sp}(n+1)/(\Z_2 \times \Sp(n-m)))_\Q \cong \Q$ and $\pi_3(\Ffour)_\Q \cong \pi_{11}(\Ffour)_\Q \cong \Q$.
\end{theorem}
\begin{proof}
    From Example \ref{Exmpl: Minimal model of Lie groups} and Lemma \ref{lem: h-groups mapping spaces}, we see that rational homotopy groups of the left and the right hand side are (abstractly) isomorphic in the claimed degrees and that the target groups are zero otherwise.
    The kernel can now easily be read of from Example \ref{Exmpl: Minimal model of Lie groups}.
    It remains to show that the inclusions a.k.a the adjoint of the symmetry group actions induce these isomorphisms.

    We spell the proof in detail for the fourth case, that is for the map $\Ffour/\Spin(7) \rightarrow \mathrm{C}(\OP^1,\OP^2)$, but we do in a way that allows a straightforward adaptation to the other cases.
    For $k = 2,3$, we have $\pi_{8k-1}(\Ffour)_\Q \cong \Q$, and by Theorem \ref{thm: homotopy classification} we see that a generator $f_{8k-1}$ can be modelled by the dga-homomorphism $\mathsf{M}(f_{8k-1}) =  a_{8k-1} \otimes \xi^\vee_{8k-1} \colon H(\Ffour) \rightarrow \mathsf{M}_{S^{8k-1}} = \Lambda[a_{8k-1}]$ that sends $\xi_{8k-1}$ to $a_{8k-1}$ and the other generators to zero.
    
    By Theorem \ref{thm: localisation of mapping spaces}, composition with the rationalisation-map $\mathsf{loc}_\Q \colon \OP^2 \rightarrow \OP^2_\Q$ induces a rationalisation $\mathrm{C}(\OP^1,\OP^2;\incl) \rightarrow \mathrm{C}(\OP^1,\OP^2_\Q;\mathsf{loc}_\Q\circ \incl)$.
    By the exponential law and the universal property of a $\Q$-localisation, each composition
    \begin{equation*}
        \xymatrix{ S^{8k-1} \ar[r]^{f} & \Ffour \ar@{^{(}->}[r]^-s & \mathrm{C}(\OP^1,\OP^2;\incl) \ar[r]^-{\mathsf{loc}_\Q\circ(\placeholder)} & \mathrm{C}(\OP^1,\OP^2_\Q;\mathsf{loc}_\Q \circ \incl)  }
    \end{equation*}
    corresponds to the following composition
    \begin{equation*}
        \xymatrix{ S^{8k-1}_\Q \times \OP^1_\Q \ar[rr]^{f_{\Q} \times \incl_\Q} && \mathrm{F}_{4,\Q} \times \OP^1_\Q \ar[rr]^{\ell_\Q} && \OP^2_\Q. }
    \end{equation*}
    Its homotopy class is modelled by the homotopy class of the composition of dga-homomorphism
    \begin{equation*}
        \xymatrix{  \mathsf{M}_{\OP^2} \ar[rr]^-{\mathsf{M}(\ell)} &&  H(\Ffour) \otimes \mathsf{M}_{\OP^1} \ar[rr]^-{\mathsf{M}(f) \otimes \id} && \mathsf{M}_{S^{8k-1}} \otimes \mathsf{M}_{\OP^1}. }
    \end{equation*}
    
    From Proposition \ref{prop: left action model projectives spaces} we deduce that, under this identification, the composition $\mathsf{loc}_\Q \circ s \circ f_{8k-1}$ is represented by the dga-homomorphism determined by $a_8 \mapsto 1 \otimes a_8$ and $b_{23} \mapsto 1 \otimes a_8b_{15} + a_{8k-1} \otimes a_8^{3-k}$, while the composition with the constant map $S^{8k-1} \xrightarrow{\mathrm{const}_1} \Ffour \rightarrow \mathrm{C}(\OP^1,\OP^2_\Q;\mathsf{loc}_\Q\circ \incl)$ is represented by the homotopy class of the dga-homomorphism given by $a_{8} \mapsto 1 \otimes a_8$ and $b_{23} \mapsto 1 \otimes a_8b_{15}$, the rational model of the inclusion $\incl \colon \OP^1 \hookrightarrow \OP^2$.
    For $k=2,3$, the elements $a_8^{3-k}$ are not exact in $\mathsf{M}_{\OP^1}$, and we deduce from Proposition \ref{prop: description of homotopies} that the two dga-homomorphisms are \emph{not} homotopic.

    We conclude that $s \colon \Ffour \rightarrow \mathrm{C}(\OP^1,\OP^2)$ induces a non-zero map between their rational homotopy groups in all degrees where the target is non-zero; by dimension reasons the map is therefore an isomorphism.
    Since $s$ factors through $\Ffour/\mathrm{Stab}(\OP^1) = \Ffour/\Spin(7)$, which has isomorphic rational homotopy groups to $\Ffour$ in these degrees, the induced map remains an isomorphism.
\end{proof}
Next we discuss the homotopy automorphisms of spheres.
\begin{theorem}\label{thm: general main result rhg different spheres}
    For $1 \leq m < n$ the map $s \colon \SOrth(n+1) \rightarrow \mathrm{C}(S^m,S^n)$ is rationally homotopic to the composition $\SOrth(n+1) \xrightarrow{\mathrm{ev}_1} S^n \hookrightarrow \mathrm{C}(S^m,S^n)$.
    In particular, the induced homomorphism $\pi_r(s)$ on rational homotopy groups is non-zero unless $r=n$, in which case the image is one-dimensional.
\end{theorem}
\begin{proof}
     Using the exponential law together with Theorem \ref{thm: homotopy classification}, we have the bijections of homotopy classes 
    \begin{equation*}
        [ \mathrm{SO}(n+1)_\Q,\mathrm{C}(S^m,S^n)_\Q ] = [ \mathrm{SO}(n+1)_\Q \times S^m_\Q, S^n_\Q ] = [ \mathsf{M}_{S^n}, H(\mathrm{SO}(n+1)) \otimes \mathsf{M}_{S^m} ]
    \end{equation*}
    under which the rationalisation $s_\Q \colon \mathrm{SO}(n+1) \rightarrow \mathrm{C}(S^m, S^n)_\Q$ corresponds to the rationalisation of the left action $\mathsf{\ell}_\Q \colon \mathrm{SO}(n+1)_\Q \times S^m_\Q \rightarrow S^n_\Q$ which is homotopic to the rationalisation evaluation map $\mathrm{ev}_{e_1,\Q}$, see also Remark \ref{rem: reduction to evaluation}.
    The statement now follows from Lemma \ref{lem: Sphere projection}.
\end{proof}
 In remains to compare the symmetry group $\SOrth(n+1)$ to the homotopy automorphisms $\hAut(S^n)$.
\begin{theorem}\label{thm: general main result rhg spheres}
   For all $1 \leq n$, the homomorphism $\pi_r(s) \colon \pi_r(\SOrth(n+1))_\Q \rightarrow \pi_r(\hAut(S^n))_\Q$ induced by canonical inclusion $s \colon \SOrth(n+1) \rightarrow \hAut(S^n)$ is the zero map unless $r=n$, in which case it is an isomorphism.
\end{theorem}
\begin{proof}
    As in the proof of Theorem \ref{thm: general main result rhg projective spaces}, whether or not the map $S^r \rightarrow \SOrth(n+1) \xrightarrow{s} \hAut(S^n)$ represents the zero element in $\pi_r(\hAut(S^n))_\Q$ can be read off from the homotopy class of the dga-homomorphism
    \begin{equation*}
        \xymatrix{ \mathsf{M}_{S^n}  \ar[rr]^-{\mathsf{M}(\ell)} && H(\SOrth(n+1)) \otimes \mathsf{M}_{S^n} \ar[rr]^-{\mathsf{M}(f) \otimes \id} && \mathsf{M}_{S^r} \otimes \mathsf{M}_{S^n}.}
    \end{equation*}
    
    If $n$ is odd, then by Proposition \ref{prop: left action model spheres} the composition is the dga-homomorphism determined by $a_n \mapsto 1 \otimes a_n + \mathsf{M}(f)(\varepsilon_{n}) \otimes 1$, which is not homotopic to the identity if $[f] \in \pi_n(\SOrth(n+1))_\Q$ represents the element $\varepsilon_n^\vee$.

    If $n$ is even, then by Proposition \ref{prop: left action model spheres} the composition is the dga-homomorphism determined by $a_n \mapsto 1 \otimes a_n$ and $b_{2n-1} \mapsto 1 \otimes b_{2n-1} + \mathsf{M}(f)(\pi_{2n-1}) \otimes 1$, which is not homotopic to the identity if $[f] \in \pi_{2n-1}(\SOrth(n+1))_\Q$ has the linear map $\pi_{2n-1}^\vee$ as rational model.
\end{proof}





\printbibliography

\end{document}

T.~Hertl acknowledges support by the Australian Research Council Discovery Project DP DP220102163.